\UseRawInputEncoding
\documentclass[12pt]{article}

\usepackage{dsfont}
\usepackage{amsfonts,amsmath,amsthm}
\usepackage{algorithm, algorithmic}
\usepackage{color}
\usepackage{graphicx}
\usepackage{stmaryrd}        

\usepackage[paper=a4paper,dvips,top=2cm,left=2cm,right=2cm,
    foot=1cm,bottom=4cm]{geometry}

\newcommand{\ve}{{\bf e}}

\begin{document}
\large

\title{Standard Dual Quaternion Optimization and Its Applications in Hand-Eye Calibration and SLAM}
\author{ Liqun Qi\footnote{Department of Applied Mathematics, The Hong Kong Polytechnic University, Hung Hom, Kowloon, Hong Kong;
    Department of Mathematics, School of Science, Hangzhou Dianzi University, Hangzhou 310018 China
    ({\tt maqilq@polyu.edu.hk}).}
}
\date{\today}
\maketitle

\begin{abstract}
Several common dual quaternion functions, such as the power function, the magnitude function, the $2$-norm function and the $k$th largest eigenvalue of a dual quaternion Hermitian matrix, are standard dual quaternion functions, i.e., the standard parts of their function values depend upon only the standard parts of their dual quaternion variables.  Furthermore, the sum, product, minimum, maximum and composite functions of two standard dual functions, the logarithm and the exponential of standard unit dual quaternion functions, are still standard dual quaternion functions. On the other hand, the dual quaternion optimization problem, where objective and constraint function values are dual numbers but variables are dual quaternions, naturally arises from applications.     We show that to solve an equality constrained dual quaternion optimization problem, we only need to solve two quaternion optimization problems. If the involved dual quaternion functions are all standard, the optimization problem is called a standard dual quaternion optimization problem, and some better results hold.  Then, we show that the dual quaternion optimization problems arising from the hand-eye calibration problem and the simultaneous localization and mapping (SLAM) problem are equality constrained standard dual quaternion optimization problems.

\medskip



\end{abstract}

\renewcommand{\Re}{\mathds{R}}
\newcommand{\rank}{\mathrm{rank}}
\renewcommand{\span}{\mathrm{span}}
\newcommand{\X}{\mathcal{X}}
\newcommand{\A}{\mathcal{A}}
\newcommand{\I}{\mathcal{I}}
\newcommand{\B}{\mathcal{B}}
\newcommand{\C}{\mathcal{C}}
\newcommand{\OO}{\mathcal{O}}
\newcommand{\e}{\mathbf{e}}
\newcommand{\0}{\mathbf{0}}
\newcommand{\dd}{\mathbf{d}}
\newcommand{\ii}{\mathbf{i}}
\newcommand{\jj}{\mathbf{j}}
\newcommand{\kk}{\mathbf{k}}
\newcommand{\va}{\mathbf{a}}
\newcommand{\vb}{\mathbf{b}}
\newcommand{\vc}{\mathbf{c}}
\newcommand{\vg}{\mathbf{g}}
\newcommand{\vr}{\mathbf{r}}
\newcommand{\vt}{\rm{vec}}
\newcommand{\vx}{\mathbf{x}}
\newcommand{\vy}{\mathbf{y}}
\newcommand{\vu}{\mathbf{u}}
\newcommand{\vv}{\mathbf{v}}
\newcommand{\y}{\mathbf{y}}
\newcommand{\vz}{\mathbf{z}}
\newcommand{\T}{\top}

\newtheorem{Thm}{Theorem}[section]
\newtheorem{Def}[Thm]{Definition}
\newtheorem{Ass}[Thm]{Assumption}
\newtheorem{Lem}[Thm]{Lemma}
\newtheorem{Prop}[Thm]{Proposition}
\newtheorem{Cor}[Thm]{Corollary}
\newtheorem{example}[Thm]{Example}
\newtheorem{remark}[Thm]{Remark}

\section{Introduction}
Dual quaternions have wide applications in robotics, 3D motion modelling and control, and computer graphics \cite{ACVL17, BK20, BLH19, CKJC16, Da99, Ke12, LLB13, LWW10, LLDL18, PSG19, WYL12}.

According to \cite{QLY22}, the magnitudes of dual quaternions and the $2$-norms of dual quaternion vectors are dual numbers.  A total order was defined for dual numbers in \cite{QLY22}.   Thus, dual quaternion optimization problems, where objective and constraint function values are dual numbers but variables are dual quaternions, naturally arise. In particular,  the least squares approach results in such optimization problems.   In Section 5, we will present such dual quaternion optimization problems arising from hand-eye calibration.    However, how to solve a dual quaternion optimization problem is still a problem.    First, there is no calculus concepts for dual quaternion functions.   Even the limit concept is not established for dual quaternion functions.  Second, unlike quaternion optimization problems, which can be transformed to one-level real optimization problems, general dual quaternion optimization problems are essentially bilevel optimization problems.  These are difficulties.

Fortunately, some common dual quaternion functions are easy to handle.   Several common dual quaternion functions, such as the power function, the magnitude function, the $2$-norm function and the $k$th largest eigenvalue of a dual quaternion Hermitian matrix, are standard dual quaternion functions.   The standard parts of the function values of these functions depend upon only the standard parts of their dual quaternion variables.  Furthermore, the sum, product, minimum, maximum and composite functions of two standard dual functions, the logarithm and the exponential of a standard unit dual quaternion functions, are still standard dual quaternion functions.  Then we show that to solve an equality constrained dual quaternion optimization problem, we only need to solve two quaternion optimization problems.   For a dual quaternion optimization problem, if its objective and constraint functions are all standard, then we call it a standard dual quaternion optimization problem, and some better results hold there.   Finally, we show that the dual quaternion optimization problems arising from the hand-eye calibration problem are equality constrained standard dual quaternion optimization problems, and thus relatively easier to handle.  This provides a method to solve this problem.

In the next section, we present some preliminary knowledge on dual numbers, dual quaternions and quaternion optimization.  Standard dual quaternion functions are studied in Section 3.   A solution method for the equality constrained dual quaternion optimization problem is proposed, and standard dual quaternion optimization is discussed in Section 4.     In Section 5, we show that equality constrained standard dual quaternion optimization problem has some better properties.   Finally, in Sections 6 and 7,
we show that the dual quaternion optimization problems arising from the hand-eye calibration problem and the simultaneous localization and mapping (SLAM) problem are equality constrained standard dual quaternion optimization problems.

Scalars, vectors and matrices are denoted by small letters, bold small letters and capital letters, respectively.   Dual numbers and dual quaternions are distinguished by a hat symbol.

\section{Dual Numbers, Quaternions and Dual Quaternions}

\subsection{Dual Numbers}

The set of the dual numbers is denoted as $\hat {\mathbb R}$.  Following the literature such as \cite{WYL12}, we use the hat symbol to denote dual numbers and dual quaternions.   A dual number $\hat q$ has the form $\hat q = q + q_d\epsilon$, where $q$ and $q_d$ are real numbers,  and $\epsilon$ is the infinitesimal unit, satisfying $\epsilon^2 = 0$.   The quaternion $q$ is called the real part or the standard part of $\hat q$, and the quaternion $q_d$ is called the dual part or the infinitesimal part of $\hat q$.  The infinitesimal unit $\epsilon$ is commutative in multiplication with real numbers, complex numbers and quaternion numbers.  If $q \not = 0$,  $\hat q$ is said to be appreciable, otherwise,  $\hat q$ is said to be infinitesimal.  

 A total order was introduced in \cite{QLY22} for dual numbers.  Given two dual numbers $\hat p, \hat q \in \hat {\mathbb R}$, $\hat p = p + p_d\epsilon$, $\hat q = q + q_d\epsilon$, where $p$, $p_d$, $q$ and $q_d$ are real numbers, we say that $\hat p \le \hat q$, if either $p < q$, or $p = q$ and $p_d \le q_d$.  In particular, we say that $\hat p$ is positive, nonnegative, nonpositive or negative, if $\hat p > 0$, $\hat p \ge 0$, $\hat p \le 0$ or $\hat p < 0$, respectively.


\subsection{Quaternions}
The set of the quaternions is denoted by $\mathbb Q$. A quaternion $q$ has the form
$q = q_0 + q_1\ii + q_2\jj + q_3\kk,$
where $q_0, q_1, q_2$ and $q_3$ are real numbers, $\ii, \jj$ and $\kk$ are three imaginary units of quaternions, satisfying
$$\ii^2 = \jj^2 = \kk^2 =\ii\jj\kk = -1, ~~\ii\jj = -\jj\ii = \kk, ~~ \jj\kk = - \kk\jj = \ii, ~~\kk\ii = -\ii\kk = \jj.$$
The real part of $q$ is Re$(q) = q_0$.   The imaginary part of $q$ is Im$(q) = q_1\ii + q_2\jj +q_3\kk$.
The multiplication of quaternions satisfies the distribution law, but is noncommutative.

The conjugate of $q = q_0 + q_1\ii + q_2\jj + q_3\kk$ is
$q^* := q_0 - q_1\ii - q_2\jj - q_3\kk.$
The magnitude of $q$ is
$|q| = \sqrt{q_0^2+q_1^2+q_2^2+q_3^2}.$
It follows that the inverse of a nonzero quaternion $q$ is
$q^{-1} = {q^* / |q|^2}.$ For any two quaternions $p$ and $q$, we have $(pq)^* = q^*p^*$.

A quaternion is said imaginary if its real part is zero.  If $q$ is imaginary, then $q^* = -q$.  In the literature \cite{WYL12}, it is called a vector quaternion.   Various 3D vectors, such as position vectors, displacement vectors, linear velocity vectors, and angular velocity vectors, can be represented as imaginary quaternions.   

If $|q|=1$, then $q$ is called a unit quaternion, or a rotation quaternion.  A spatial rotation around a fixed point of $\theta$  radians about a unit axis $(x_1,x_2,x_3)$ that denotes the Euler axis is given by the unit quaternion 
\begin{equation} \label{uq}
q = \cos(\theta /2) + x_1\sin(\theta /2)\ii + x_2\sin(\theta /2)\jj + x_3\sin(\theta /2)\kk = e^{{\theta \over 2}x}.
\end{equation}
where the unit axis $x$ is an imaginary unit quaternion $x = x_1\ii + x_2\jj + x_3\kk$.   We may also write
\begin{equation} \label{lnuq}
\ln q = {\theta \over 2}x.
\end{equation}
A unit quaternion $q$ is always invertible and $q^{-1} = q^*$.


The collection of $n$-dimensional quaternion  vectors is denoted by ${\mathbb {Q}}^n$. For $\vx = (x_1, x_2,\cdots, x_n)^\top, \vy = (y_1, y_2,\cdots, y_n)^\top  \in {\mathbb {Q}}^n$, define $\vx^*\vy = \sum_{i=1}^n x_i^*y_i$, where $\vx^* = (x_1^*, x_2^*,\cdots, x_n^*)$ is the conjugate transpose of $\vx$.

\subsection{Dual Quaternions and Unit Dual Quaternions}

The set of dual quaternions is denoted by $\hat {\mathbb Q}$.   A dual quaternion $\hat q \in \hat {\mathbb Q}$ has the form
\begin{equation} \label{dq}
\hat q = q + q_d\epsilon,
\end{equation}
where $q, q_d \in \mathbb {Q}$ are the standard part and the dual part of $\hat q$, respectively. If $q \not = 0$, then we say that $\hat q$ is appreciable.  If $q$ and $q_d$ are imaginary quaternions, then $\hat q$ is called an imaginary dual quaternion.


The conjugate of $\hat q$ is
\begin{equation} \label{conjugate}
\hat q^* = q^* + q_d^*\epsilon.
\end{equation}
Thus, if $\hat q = \hat q^*$, then $\hat q$ is a dual number.  If $\hat q$ is imaginary, then $\hat q^* = - \hat q$.

The magnitude of $\hat q$ was defined in \cite{QLY22} as
\begin{equation} \label{magnitude}
|\hat q| := \left\{ \begin{aligned} |q| + {(qq_d^*+q_d q^*) \over 2|q|}\epsilon, & \ {\rm if}\  q \not = 0, \\
|q_d|\epsilon, &  \ {\rm otherwise},
\end{aligned} \right.
\end{equation}
which is a dual number.

For two dual quaternions $\hat p = p + p_d \epsilon$ and $\hat q = q + q_d \epsilon$, their addition and multiplications are defined as
$$\hat p+\hat q = \left(p + q\right) + \left(p_d + q_d\right)\epsilon$$
and
$$\hat p\hat q = pq + \left(pq_d + p_d q\right)\epsilon.$$
A dual number is always commutative with a dual quaternion or a dual quaternion vector.

A dual quaternion $\hat q$ is called invertible if there exists a quaternion $\hat p$ such that $\hat p\hat q = \hat q\hat p =1$.  A dual quaternion $\hat q$ is invertible if and only if  $\hat q$ is appreciable. In this case, we have
$$\hat q^{-1} = q^{-1} - q^{-1}q_d q^{-1} \epsilon.$$


If $|\hat q| = 1$, then $\hat q$ is called a unit dual quaternion. A unit dual quaternion $\hat q$ is always invertible and we have ${\hat q}^{-1} = {\hat q}^*$.   The 3D motion of a rigid body can be represented by a unit dual quaternion.
We have
$$\hat q\hat q^* = (q + q_d\epsilon)(q^* + q_d^*\epsilon)= qq^* + (qq_d^* + q_d q^*)\epsilon = \hat q^*\hat q.$$
Thus, $\hat q$ is a unit dual quaternion if and only if $q$ is a unit quaternion, and
\begin{equation} \label{udq1}
qq_d^* + q_d q^* = q^*q_d + q_d^* q=0.
\end{equation}
Suppose that there is a rotation $q \in {\mathbb Q}$ succeeded by a translation $p^b \in {\mathbb Q}$, where $p^b$ is an imaginary quaternion.   The superscripts $b$ and $s$ represent the relation of the rigid body motion with respect to the body frame attached to the rigid body and the spatial frame which is relative to a fixed coordinate frame.
Then the whole transformation can be represented by unit dual quaternion $\hat q = q + q_d \epsilon$, where $q_d = {1 \over 2}qp^b$.   Note that we have
$$qq_d^* + q_d q^* = {1 \over 2}\left[q(p^b)^*q^* + qp^bq^*\right] = {1 \over 2}q\left[(p^b)^*+p^b\right]q^* = 0.$$
Thus, a transformation of a rigid body can be represented by a unit dual quaternion
\begin{equation} \label{udq}
\hat q = q + {\epsilon \over 2}qp^b,
\end{equation}
where $q$ is a unit quaternion to represent the rotation, and $p^b$ is the imaginary quaternion to represent the translation or the position.   
In (\ref{udq}), $q$ is the attitude of the rigid body, while $\hat q$ represents the transformation.


Combining (\ref{udq}) with (\ref{lnuq}), we have
\begin{equation} \label{lnudq}
\ln \hat q = {1 \over 2}(\theta x + \epsilon p^b).
\end{equation}


A unit dual quaternion $\hat q$ serves as both a specification of the configuration of a rigid body and a transformation taking the coordinates of a point from one frame to another via rotation and translation.
In (\ref{udq}), if $\hat q$ is the configuration of the rigid body, then $q$ and $p^b$ are the attitude of and position of the rigid body respectively.




Denote the collection of $n$-dimensional dual quaternion  vectors by ${\hat {\mathbb Q}}^n$.

For $\hat \vx = (\hat x_1, \hat x_2,\cdots, \hat x_n)^\top$, $\hat \vy = (\hat y_1, \hat y_2,\cdots, \hat y_n)^\top  \in {\hat {\mathbb Q}}^n$, define $\hat \vx^*\hat \vy = \sum_{i=1}^n \hat x_i^*\hat y_i$, where $\hat \vx^* = (\hat x_1^*, \hat x_2^*,\cdots, \hat x_n^*)$ is the conjugate transpose of $\hat \vx$.   We say $\hat \vx$ is appreciable if at least one of its component is appreciable.  

For $\hat \vx = (\hat x_1, \hat x_2,\cdots, \hat x_n)^\top$, by \cite{QLY22}, if not all of $\hat x_i$ are infinitesimal, its $2$-norm is defined as
\begin{equation} \label{e9}
\|\hat \vx\|_2 = \sqrt{\sum_{i=1}^n |\hat x_i|^2}.
\end{equation}
If all $\hat x_i$ are infinitesimal, we have $\hat x_i = (x_i)_d \epsilon$ for $i = 1, 2,\ldots, n$.  Then we have
\begin{equation} \label{e10}
\|\hat \vx\|_2 = \sqrt{\sum_{i=1}^n |(x_i)_d |^2} \epsilon.
\end{equation}

\section{Standard Dual Quaternion Functions}

We call a function $\hat f : {\hat {\mathbb Q}}^n \to {\hat {\mathbb R}}$ a dual quaternion function.   Let
$$\hat f(\hat \vx) = f(\hat \vx) + \epsilon f_d(\hat \vx)$$
and
$$\hat \vx = \vx + \epsilon \vx_d,$$
where $f(\hat \vx), f_d(\hat \vx) \in \mathbb Q$, $\vx, \vx_d \in {\mathbb Q}^n$ are the standard parts and dual parts of $\hat f(\hat \vx)$ and $\hat \vx$, respectively.  If $f(\hat \vx) \equiv f(\vx)$ for all $\hat \vx \in {\hat {\mathbb Q}}^n$, then we say that $\hat f$ is a {\bf standard dual quaternion function}.

A simple example is the power function.   Consider the power function $\hat f(\hat x) = (\hat x)^m$ for a positive integer $m$. Let $\hat x = x + \epsilon x_d$.  Then $\hat f(\hat x) = (\hat x)^m = x^m +m\epsilon x^{m-1}x_d$, i.e., $f(\hat x) = x^m = \hat f(x)$.  By (\ref{magnitude}), the magnitude function is also a standard quaternion function.    By (\ref{magnitude}), (\ref{e9}) and (\ref{e10}), the $2$-norm function is also a standard quaternion function.   A fourth example is $\hat f(\hat A) = \hat \lambda_k(\hat A)$, where $\hat A$ is the an $n \times n$ dual quaternion Hermitian matrix, $\hat \lambda_k(\hat A)$ is the $k$th largest eigenvalue of $A$, $1 \le k \le n$.   By Theorem 4.1 of \cite{QL22}, this is also a standard dual quaternion function.

Furthermore, many operations preserve standard dual quaternion functions.

\begin{Thm} \label{t3.1}
Suppose that $\hat f, \hat g : {\hat {\mathbb Q}}^n \to {\hat {\mathbb R}}$ are two standard dual quaternion functions.    Then their sum, product, minimum and maximum functions are still standard dual quaternion functions.
\end{Thm}
\begin{proof}  Let $\hat h(\hat \vx) = \hat f(\hat \vx) + \hat g(\hat \vx)$.  Then $\hat h = h + \epsilon h_d$, where
$$h(\hat \vx) = f(\hat \vx) + g(\hat \vx) = f(\vx)+ g(\vx) = h(\vx),$$
i.e., $\hat h$ is also a standard dual quaternion function.  This proves the first conclusion.

Let $\hat h(\hat \vx) = \hat f(\hat \vx)\hat g(\hat \vx)$.  Then $\hat h = h + \epsilon h_d$, where
$$h(\hat \vx) = f(\hat \vx)g(\hat \vx) = f(\vx)g(\vx) = h(\vx),$$
i.e., $\hat h$ is also a standard dual quaternion function.  This proves the second conclusion.

Let $\hat h(\hat \vx) = \min \{ \hat f(\hat \vx), \hat g(\hat \vx) \}$.  Then $\hat h = h + \epsilon h_d$, where
$$h(\hat \vx) = \min \{ f(\hat \vx), g(\hat \vx) \} = \min \{ f(\vx), g(\vx) \} = h(\vx),$$
i.e., $\hat h$ is also a standard dual quaternion function.  This proves the third conclusion.

The fourth conclusion can be proved similarly.

\end{proof}

\begin{Cor} \label{c3.2}
Suppose that $\hat f_1, \cdots, \hat f_m: {\hat {\mathbb Q}}^n \to {\hat {\mathbb R}}$ are $m$ standard dual quaternion functions, where $m$ is a positive integer.    Then their sum, product, minimum and maximum functions are still standard dual quaternion functions.
\end{Cor}

\begin{Cor} \label{c3.3}
Suppose that $\hat f: {\hat {\mathbb Q}}^n \to {\hat {\mathbb R}}$ is a standard dual quaternion functions, and $m$ is a positive integer.    Then $(\hat f)^m$ is still a standard dual quaternion function.
\end{Cor}

\begin{Thm} \label{t3.4}
Suppose that $\hat f: \hat {\mathbb Q} \to {\hat {\mathbb R}}$ is a standard unit dual quaternion function.     Then its logarithm and exponential functions are still standard dual quaternion functions.
\end{Thm}
\begin{proof} We have
$$\hat f(\hat x) = f(\hat x) + \epsilon f_d(\hat x) = f(x) + \epsilon f_d(\hat x),$$
as $\hat f$ is a standard dual quaternion function.  Since $f$ is a unit dual quaternion function, we have $|f(x)|=1$.  Thus, we may write
$$\hat f(\hat x) = f(x) + {\epsilon \over 2}f(x)p^b.$$
Write
$$f(x) = \cos(\theta /2) + y_1\sin(\theta /2)\ii + y_2\sin(\theta /2)\jj + y_3\sin(\theta /2)\kk.$$
Then $\theta /2$ and $y = y_1\ii + y_2\jj + y_3\kk$ are functions of $x$.
By (\ref{lnudq}), we have
$$\ln \hat f(\hat x) = {1 \over 2}(\theta y + \epsilon p^b).$$
Thus, $\ln \hat f$ is still a standard dual quaternion function.   Similarly, we may show that $e^{\hat f}$ is a standard dual quaternion function too.
\end{proof}
\begin{Thm} \label{t3.5}
Suppose that $\hat f: \hat {\mathbb Q} \to {\hat {\mathbb R}}$ and $\hat g: \hat {\mathbb Q}^n \to {\hat {\mathbb R}}$ are two standard unit dual quaternion functions.    Then their composite function $\hat h = \hat f \circ g :  \hat {\mathbb Q}^n \to {\hat {\mathbb R}}$ is also a standard dual quaternion function.
\end{Thm}
\begin{proof}  For $\hat \vx \in {\mathbb Q}^n$,
$$h(\hat \vx) = f(\hat g(\hat \vx)) = f(g(\hat \vx)) = f(g(\vx)) = h(\vx),$$
where the second and the third equalities hold because $\hat f$ and $\hat g$ are standard unit dual quaternion functions respectively.   Thus, $\hat h$ is also a standard dual quaternion function.
\end{proof}

\section{Equality Constrained Dual Quaternion Optimization}

In this section, we propose a solution method for the following equality constrained dual quaternion optimization problem ({\bf EQDQO}):
\begin{equation} \label{e22a}
\min \{ \hat f(\hat \vx) :  \hat h_j(\hat \vx) = 0, \ j = 1, \cdots, m  \},
\end{equation}
where $\hat \vx \in {\hat {\mathbb Q}}^n$,  $\hat f$ and $\hat h_j$ for $j = 1, \cdots, m$ are dual quaternion functions.  If $\hat f$ and $\hat h_j$ for $j = 1, \cdots, m$ are all standard dual quaternion functions, then this problem is called an equality constrained standard dual quaternion optimization problem ({\bf EQSDQO}).
 {\bf EQSDQO} arises in many applications, as the most dual quaternion application problems are unit dual quaternion application problems, and their constraints may only have the unit requirement.   In the next section, we will see that the dual quaternion optimization problem arising from the hand-eye calibration problem is an {\bf EQSDQO}.

 We have $\hat f(\hat \vx) = f(\vx + \vx_d\epsilon) + f_d(\vx + \vx_d\epsilon)\epsilon$.  We may regard $f$ and $f_d$ as real valued functions $\bar f$ and $\bar f_d$, with $2n$-dimensional quaternion vector variables $(\vx, \vx_d)$, i.e.,
 $$\bar f(\vx, \vx_d) \equiv f(\vx + \vx_d\epsilon), \ \bar f_d(\vx, \vx_d) \equiv f_d(\vx + \vx_d\epsilon).$$
 We now abuse the notation, and simply write
 $$f(\vx, \vx_d) \equiv f(\vx + \vx_d\epsilon), \ f_d(\vx, \vx_d) \equiv f_d(\vx + \vx_d\epsilon),$$
 i.e., we now simply regard $f$ and $f_d$ as real valued functions, with $2n$-dimensional quaternion vector variables $(\vx, \vx_d)$.   We may treat $h_j$ and $(h_j)_d$ similarly.   Then we have the following theorem.

\begin{Thm} \label{t4.1}
Suppose that $L^{I}$ is the optimal function value of the equality constrained quaternion optimization problem ({\bf EQQOPI}):
\begin{equation} \label{e12}
\min \{ f(\vx, \vx_d) : h_j(\vx, \vx_d) = 0, (h_j)_d(\vx, \vx_d) = 0, \ j = 1, \cdots, m  \},
\end{equation}
and $(\vx^{opt}, \vx_d^{opt})$ is a global optimal solution of the following equality constrained quaternion optimization problem ({\bf EQQOPII}):
\begin{equation} \label{e13}
\min \{ f_d(\vx, \vx_d) : f(\vx, \vx_d) = L^{I},  \ h_j(\vx, \vx_d) = 0, \ (h_j)_d(\vx, \vx_d) = 0, \ j = 1, \cdots, m  \}.
\end{equation}
Then $\hat \vx^{opt} = \vx^{opt} + \epsilon \vx_d^{opt}$ is a global optimal solution of {\bf EQDQO} (\ref{e22a}).
\end{Thm}
\begin{proof}
Since $\hat \vx^{opt} = \vx^{opt} + \epsilon \vx_d^{opt}$ satisfies all the constraints of  {\bf EQDQO}, and the standard part and the dual part of $f$ attain their global minimal values, the conclusion follows.
\end{proof}

Let $\vx = \vx_0 + \vx_1\ii + \vx_2\jj + \vx_3\kk$ and $\vx_d = \vx_4 + \vx_5\ii + \vx_6\jj + \vx_7\kk$.  Then we may further regard $f(\vx, \vx_d)$, $f_d(\vx, \vx_d)$, $h_j(\vx, \vx_d)$ and $(h_j)_d(\vx, \vx_d)$ as real valued functions with real vector variables $\vx_k$ for $k = 0, \cdots, 7$.  Then we only need to apply ordinary nonlinear optimization techniques to analyze and to solve (\ref{e12}-\ref{e13}).   However, we may apply the approach in \cite{QLWZ22} to simplify the language.

In the following, as in \cite{QLWZ22}, we say that the real valued functions $f, f_d, h_j, (h_j)_d$ are continuously differentiable, if they are continuously differentiable with respect to real vector variables $\vx_k$ for $k = 0, \cdots, 7$.   We now take $f$ as an example, and the same argument applies to $f_d, h_j$ and $(h_j)_d$.  Let $\mathbb H = {\mathbb Q}^n \times {\mathbb Q}^n$.  Denote the partial derivatives and gradient of $f$ at $(\vx, \vx_d)$ as
$${\partial \over \partial \vx}f(\vx, \vx_d) = {\partial \over \partial \vx_0}f(\vx, \vx_d) + {\partial \over \partial \vx_1}f(\vx, \vx_d) \ii + {\partial \over \partial \vx_2}f(\vx, \vx_d) \jj + {\partial \over \partial \vx_3}f(\vx, \vx_d) \kk,$$
$${\partial \over \partial \vx_d}f(\vx, \vx_d) = {\partial \over \partial \vx_4}f(\vx, \vx_d) + {\partial \over \partial \vx_5}f(\vx, \vx_d) \ii + {\partial \over \partial \vx_6}f(\vx, \vx_d) \jj + {\partial \over \partial \vx_7}f(\vx, \vx_d) \kk,$$
$$\nabla f(\vx, \vx_d) = \left({\partial \over \partial \vx}f(\vx, \vx_d), {\partial \over \partial \vx_d}f(\vx, \vx_d)\right).$$
Then $\nabla f(\vx, \vx_d) \in \mathbb H$.  With the R-linear independence concept and Theorem 4.3 of \cite{QLWZ22}, we may have the first optimality conditions for (\ref{e12}-\ref{e13}).    This is essentially the linear independence constraint qualification in real nonlinear optimization.  

\section{Equality Constrained Standard Dual Quaternion Optimization}

If $\hat f$ and $\hat h_j$ for $j = 1, \cdots, m$ are all standard dual quaternion functions, then
({\bf EQQOPI}) has the form
\begin{equation} \label{e23c}
\min \{ f(\vx) : h_j(\vx) = 0, (h_j)_d(\vx, \vx_d) = 0, \ j = 1, \cdots, m  \},
\end{equation}
and ({\bf EQQOPII}) has the form
\begin{equation} \label{e15}
\min \{ f_d(\vx, \vx_d) : f(\vx) = L^{I},  \ h_j(\vx) = 0, \ (h_j)_d(\vx, \vx_d) = 0, \ j = 1, \cdots, m  \}.
\end{equation}
Then we may have some better results.   See the theorem and its corollary below.  We call such a dual quaternion optimization problem a standard dual quaternion optimization problem.

\begin{Thm} \label{t4.2}
Suppose that $f$, $h_j$ for $j = 1, \cdots, m$ are continuously differentiable.

Assume $(\vx^{opt}, \vx_d^{opt})$ is an optimal solution of (\ref{e23c}).   If $\left\{ \nabla h_j(\vx^{opt}) : j = 1, \cdots, m  \right\}$ is R-linearly independent in the sense of \cite{QLWZ22}, and $\left\{ {\partial \over \partial \vx_d} (h_j)_d(\vx^{opt}, \vx_d^{opt}) : j = 1, \cdots, m  \right\}$ is R-linearly independent in the sense of \cite{QLWZ22}, then there are real Lagrangian multipliers $\lambda_j$ for $j = 1, \cdots, m$, such that
\begin{equation}
\nabla f(\vx^{opt}) + \sum_{j=1}^m \lambda_j \nabla h_j(\vx^{opt}) = 0,
\end{equation}
and
\begin{equation}
 h_j(\vx^{opt}) = 0, \  (h_j)_d(\vx^{opt}, \vx_d^{opt}) = 0, \ {\rm for} \ j = 1, \cdots, m.
\end{equation}

Assume $(\vx^{opt}, \vx_d^{opt})$ is an optimal solution of (\ref{e15}).   If $\left\{ \nabla f(\vx^{opt}), \nabla h_j(\vx^{opt}) : j = 1, \cdots, m  \right\}$ is R-linearly independent in the sense of \cite{QLWZ22}, and $\left\{ {\partial \over \partial \vx_d} (h_j)_d(\vx^{opt}, \vx_d^{opt}) : j = 1, \cdots, m  \right\}$ is R-linearly independent in the sense of \cite{QLWZ22}, then there are real Lagrangian multipliers $\sigma$ and $\lambda_j$ for $j = 1, \cdots, m$, such that
\begin{equation}
\nabla f_d(\vx^{opt}, \vx_d^{opt}) + \sigma \nabla f(\vx^{opt}) + \sum_{j=1}^m \lambda_j \nabla h_j(\vx^{opt}) = 0,
\end{equation}
and
\begin{equation}
 f(\vx^{opt}) = L^{I},\ h_j(\vx^{opt}) = 0, \  (h_j)_d(\vx^{opt}, \vx_d^{opt}) = 0, \ {\rm for} \ j = 1, \cdots, m.
\end{equation}
\end{Thm}
\begin{proof}  By Theorem 4.3 of \cite{QLWZ22}£¬ and the property of block triangular matrices, if $\left\{ \nabla h_j(\vx^{opt}) : j = 1, \cdots, m  \right\}$ is R-linearly independent in the sense of \cite{QLWZ22}, and $\left\{ {\partial \over \partial \vx_d} (h_j)_d(\vx^{opt}, \vx_d^{opt}) : j = 1, \cdots, m  \right\}$ is R-linearly independent in the sense of \cite{QLWZ22}, then there are real Lagrangian multipliers $\lambda_j$ and $\mu_j$ for $j = 1, \cdots, m$, such that
\begin{equation}
\nabla f(\vx^{opt}) + \sum_{j=1}^m \left[\lambda_j \nabla h_j(\vx^{opt}) + \mu_j {\partial \over \partial \vx} (h_j)_d(\vx^{opt}, \vx_d^{opt})\right] = 0,
\end{equation}
\begin{equation} \label{ea.1}
\sum_{j=1}^m \mu_j {\partial \over \partial \vx_d} (h_j)_d(\vx^{opt}, \vx_d^{opt}) = 0,
\end{equation}
and
\begin{equation}
 h_j(\vx^{opt}) = 0, \  (h_j)_d(\vx^{opt}, \vx_d^{opt}) = 0, \ {\rm for} \ j = 1, \cdots, m.
\end{equation}
Since $\left\{ {\partial \over \partial \vx_d} (h_j)_d(\vx^{opt}, \vx_d^{opt}) : j = 1, \cdots, m  \right\}$ is R-linearly independent in the sense of \cite{QLWZ22}, by (\ref{ea.1}), all $\mu_j = 0$.   We have the first result of this theorem.   The second result can be derived similarly.
\end{proof}

Actually, Theorem \ref{t4.2} may be stated and proved in the Language of real optimization.  That way is somewhat tedious.

\begin{Cor} \label{c4.3}
Let $m=1$ and denote $h \equiv h_1$.  Suppose that $f$ and $h$ are continuously differentiable.

Assume $(\vx^{opt}, \vx_d^{opt})$ is an optimal solution of (\ref{e23c}).  Assume that $ \nabla h(\vx^{opt}) \not = \0$, and ${\partial \over \partial \vx_d} h_d(\vx^{opt}, \vx_d^{opt}) \not = \0$.  Then there is a real Lagrangian multiplier $\lambda$, such that
\begin{equation}
\nabla f(\vx^{opt}) + \lambda \nabla h(\vx^{opt}) = 0, 
\end{equation}
and
\begin{equation}
 h(\vx^{opt}) = 0, \  h_d(\vx^{opt}, \vx_d^{opt}) = 0.
\end{equation}

Assume $(\vx^{opt}, \vx_d^{opt})$ is an optimal solution of (\ref{e15}).   If $\left\{ \nabla f(\vx^{opt}), \nabla h(\vx^{opt})  \right\}$ is R-linearly independent in the sense of \cite{QLWZ22}, then there are real Lagrangian multipliers $\sigma$ and $\lambda$, such that
\begin{equation}
\nabla f_d(\vx^{opt}, \vx_d^{opt}) + \sigma \nabla f(\vx^{opt}) + \lambda \nabla h(\vx^{opt}) = 0£¬  
\end{equation}
and
\begin{equation}
 f(\vx^{opt}) = L^{I},\ h(\vx^{opt}) = 0, \  h_d(\vx^{opt}, \vx_d^{opt}) = 0.
\end{equation}
\end{Cor}

In practice, we may only find an approximal value of $L^{I}$, an error analysis is necessary in the application of this approach.

This approach cannot be extended to the inequality constraint case naturally, the optimal index set there is dependent upon the concrete solution.


\section{Hand-Eye Calibration}

 The hand-eye calibration problem (the sensor-actuator problem) is an important application problem in robot research.     They can be solved by using dual quaternions \cite{Da99, LWW10, LLDL18}.  In this section, we formulate it as an equality constrained dual quaternion optimization problem, and show that this problem is a standard dual quaternion optimization problem.

In 1989, Shiu and Ahmad \cite{SA89} and Tsai and Lenz \cite{TL89} formulated the hand-eye calibration problem as a matrix equation
\begin{equation} \label{4.1}
AX=XB,
\end{equation}
where $X$ is the transformation matrix from the camera (eye) to the gripper (hand), $A = A_2A_1^{-1}$ and $B = B_2^{-1}B_1$, $A_i$ is the transformation matrix from the camera to the world coordinate system, and $B_i$ is the transformation matrix from
the robot base to the gripper \cite{Da99}.  They are all $4 \times 4$ matrices, $A_i$ and $B_i$ can be measured, while $X$ is the variable matrix.   Assume that there are $n+1$ measurements (poses) $A_i$ and $B_i$ for $i = 1, \cdots, n+1$.  Denote $A^{(i)} = A_{i+1}A_i^{-1}$ and $B^{(i)} = B_{i+1}^{-1}B_i$, for $i = 1, \cdots, n$.   Then the problem is to find the best solution $X$ of
 \begin{equation} \label{4.2}
A^{(i)}X=XB^{(i)},
\end{equation}
for $i= 1, \cdots, n$.
In 1999, Daniilidis \cite{Da99} proposed to use dual quaternions to solve this problem.   See Figure \ref{AXXB} for the geometry of this model.

\begin{figure}
  \centering
  \includegraphics[width=0.8\textwidth]{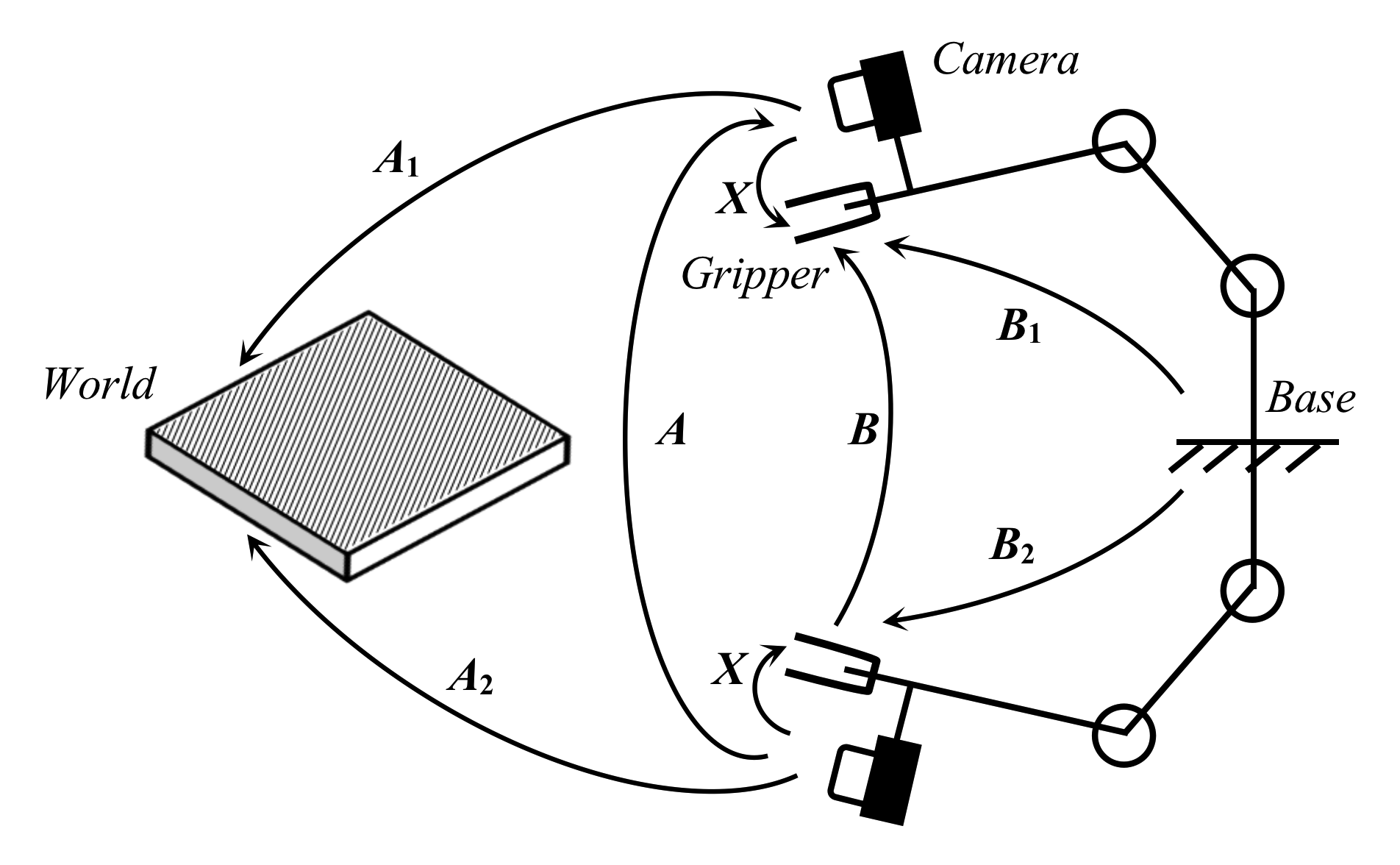}\\
  \caption{$AX=XB$ Hand-Eye Calibration Model.}\label{AXXB}
\end{figure}

In 1994, Zhuang, Roth and Sudhaker \cite{ZRS94} generalized (\ref{4.1}) to
\begin{equation} \label{4.3}
AX=YB,
\end{equation}
where $Y$ is the transformation matrix from the world coordinate system to the robot base.   Assume that we have $n$ measurements (poses) for $i = 1, \cdots, n$.   Then the problem is to find the best solution $X$ and $Y$ of
 \begin{equation} \label{4.4}
A_iX=YB_i,
\end{equation}
for $i= 1, \cdots, n$.   In 2010, Li, Wang and Wu \cite{LWW10} proposed to use dual quaternions to solve this problem.    Also see \cite{LLDL18}.    See Figure \ref{AXYB} for the geometry of this model.

\begin{figure}
  \centering
  \includegraphics[width=0.8\textwidth]{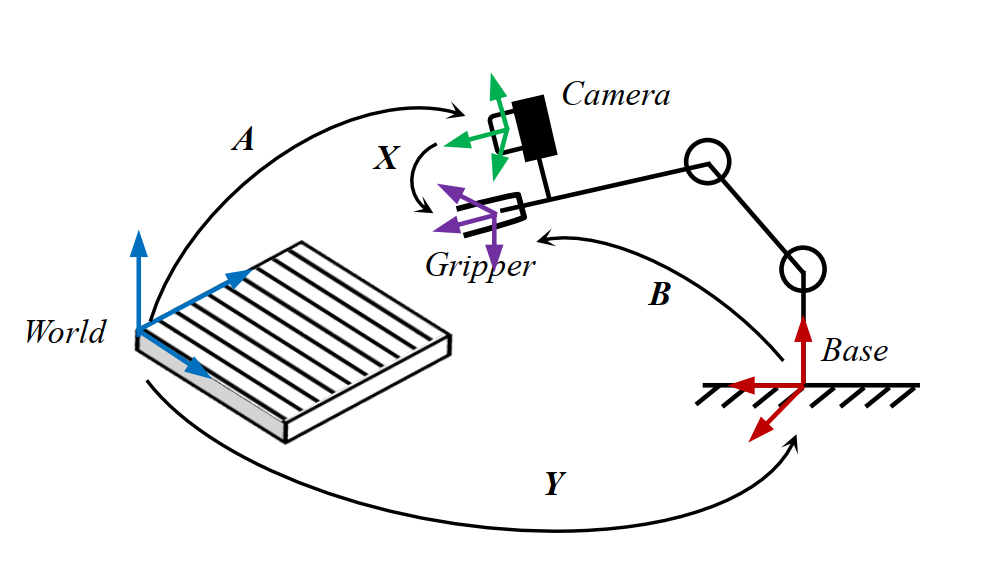}\\
  \caption{$AX=YB$ Hand-Eye Calibration Model.}\label{AXYB}
\end{figure}

As the 3D movement of a rigid body may also be expressed by unit dual quaternions, we may rewrite (\ref{4.2}) and (\ref{4.4}) as
 \begin{equation} \label{4.5}
\hat a^{(i)} \hat x = \hat x \hat b^{(i)},
\end{equation}
for $i= 1, \cdots, n$, and
 \begin{equation} \label{4.6}
\hat a_i \hat x = \hat y \hat b_i,
\end{equation}
for $i= 1, \cdots, n$, where $\hat a^{(i)}$, $\hat b^{(i)}$, $\hat a_i$, $\hat b_i$, $\hat x$ and $\hat y$ are corresponding unit dual quaternions.   We may solve (\ref{4.5}) and (\ref{4.6}) by the following minimization problems:
\begin{equation} \label{4.7}
\min \left\{ \sum_{i=1}^n | \hat a^{(i)} \hat x - \hat x \hat b^{(i)}| : |\hat x|^2 = 1 \right\},
\end{equation}
and
\begin{equation} \label{4.8}
\min \left\{ \sum_{i=1}^n | \hat a_i\hat x - \hat y \hat b_i| : |\hat x |^2 = 1,
|\hat y |^2 = 1 \right\}.
\end{equation}
Then (\ref{4.7}) and (\ref{4.8}) are equality constrained dual quaternion optimization problems.   By the analysis in Section 3, all the functions involved are standard dual quaternion optimization problems.  Thus, they are equality constrained standard dual quaternion optimization problems.

Note that by (\ref{magnitude}), the magnitudes in the objective functions of (\ref{4.7}) and (\ref{4.8}) cannot be replaced by their squares as for a dual quaternion number $\hat q$, $|\hat q | = 0$ and $|\hat q|^2 = 0$ are not equivalent.

Consider (\ref{4.7}).
We have $\hat f(\hat x) = \sum_{i=1}^n | \hat a^{(i)}\hat x - \hat x \hat b^{(i)}|$ and
$\hat h(\hat x) = |\hat x |^2 - 1$.   Then {\bf EQSDQO} has the form
\begin{equation} \label{5.1}
\min \{ \hat f(\hat x) :  \hat h(\hat x) = 0  \}.
\end{equation}
By (\ref{e9}) and (\ref{e10}),
$$h(x, x_d) = h(x) = | x |^2 - 1,$$
and
$$h_d(x, x_d) = xx_d^* + x_dx^*.$$
Similarly, we have
$$f(x, x_d) = f(x) = \sum_{i=1}^n |  a^{(i)} x -  x b^{(i)} |,$$
and $f_d(x, x_d)$ can be computed by (\ref{magnitude}).
Then,
({\bf EQQOPI}) has the form
\begin{equation}
\min \{ f(x) : h(x) = 0, h_d(x, x_d) = 0,  \},
\end{equation}
and ({\bf EQQOPII}) has the form
\begin{equation} \label{e24c}
\min \{ f_d(x, x_d) : f(x) = L^{I},  \ h(x) = 0, \ h_d(x, x_d) = 0 \}.
\end{equation}
We may apply Theorems \ref{t4.1} and \ref{t4.2} as well as Corollary \ref{c4.3} now.   In particular, we have $\nabla h(x^{opt}) = 2x^{opt} \not = 0$ as $|x^{opt}|^2 = 1$.  The problem (\ref{4.8}) can be treated similarly.

\section{Pose Graph Optimization}

The simultaneous localization and mapping (SLAM) problem is a very hot topic in robotic research \cite{CCCLSNRL16}.  The applications of SLAM include not only field robots, but also underwater navigation \cite{WND00}, unmanned aerial vehicle (UAV) \cite{BS07}, planetary exploration rover \cite{WJZ13}.  The graph-based approach poses the SLAM problem to the pose graph optimization problem.  In 2016, Cheng, Kim, Jiang and Che \cite{CKJC16} studied dual quaternion-based graph SLAM.   In this section, we show that the dual quaternion pose graph optimization problem is a standard dual quaternion optimization problem.

Pose graph optimization estimates $n$ robot poses from $m$ relative pose measurements.   The dual quaternion pose graph problem can be visualized as a directed graph $G = (V, E)$ \cite{CTDD15}, where each vertex $i \in V$ corresponds to a robot pose $\hat x_i$ for $i = 1, \cdots, n$, and each directed edge (arc) $(i, j) \in E$ corresponds to a relative measurement $\hat q_{ij}$, a unit dual quaternion.    Thus, $\hat x_i$ for $i = 1, \cdots, n$, are $n$ unknown unit dual quaternion variables.    The cardinality of the edge set $|E| = m$ as we have $m$ relative measurements.    By the pose relation, the relative pose $\hat y_{ij}$ should be
$$\hat y_{ij} = \hat x_i^* \hat x_j.$$
Then the error at the edge $(i, j)$ is
$$\hat e_{ij} = \hat q_{ij} - \hat x_i^* \hat x_j.$$

We may regard $\hat \ve = (\hat e_{ij} : (i, j) \in E)$ as an $m$-dimensional dual quaternion vector, i.e., $\hat \ve \in {\hat {\mathbb Q}}^m$.   Then we have the following minimization model for this pose graph optimization problem:
\begin{equation} \label{7.1}
\min \left\{ \|\hat \ve \|_2 : |\hat x_i|^2 = 1, i = 1, \cdots, n \right\},
\end{equation}
where $\|\hat \ve \|_2$ is defined by (\ref{e9}) and (\ref{e10}).   By the discussion in Section 3, this is an equality constrained standard dual quaternion optimization problem.

We will study the algorithm aspect of applying standard dual quaternion optimization approach to the hand-eye calibration problem and the SLAM problem  in further coming  papers.

\bigskip

{\bf Acknowledgment}  I am thankful to Chen Ling, Ziyan Luo and Zhongming Chen for the discussion on standard dual quaternion optimization, to Wei Li for the discussion on hand-eye calibration, to Jiantong Cheng for the discussion on SLAM, to Guyan Ni for introducing Jiantong Cheng to me, and to Chen Ouyang and Jinjie Liu for Figures 1 and 2.

\bigskip





\end{document}